\documentclass{amsart}
\usepackage{amssymb}

\usepackage{xcolor}
\usepackage{soul}

\newcommand\gl{\mathrm{GL}}
\newcommand\pgl{\mathrm{PGL}}
\newcommand\psl{\mathrm{PSL}}
\newcommand{\reals}{\mathbb{R}}
\newcommand{\cpx}{\mathbb{C}}
\newcommand{\integers}{\mathbb{Z}}
\newcommand{\sphere}{\mathbb{S}}

\newcommand{\proj}{\mathbb{P}}

\newcommand{\gsl}{\mathrm{SL}}
\newcommand{\limitset}{\Lambda}

\newcommand{\torus}{\mathbb{T}}
\newcommand{\semidirectprod}{\reals^N \rtimes_A \reals}

\newcommand{\abs}[1]{\lvert#1\rvert}

\newtheorem{theorem}{Theorem}[section]
\newtheorem{lemma}[theorem]{Lemma}
\newtheorem{proposition}[theorem]{Proposition}
\newtheorem{corollary}[theorem]{Corollary}

\theoremstyle{definition}

\theoremstyle{remark}
\newtheorem{remark}[theorem]{Remark}

\numberwithin{equation}{section}

\begin{document}
  \title[A family of complex Kleinian split solvable groups]{A family of complex Kleinian split solvable groups}
  
  \author{Waldemar Barrera}
  \address{School of mathematics, Autonomous University of Yucatan, Mexico}
  \email{bvargas@correo.uady.mx}
  
  \author{Ren\'e Garc\'ia}
  \address{School of mathematics, Autonomous University of Yucatan, Mexico}
  \email{rene.garcia@correo.uady.mx}
  
  \author{Juan Pablo Navarrete}
  \address{School of mathematics, Autonomous University of Yucatan, Mexico}
  \email{jp.navarrete@correo.uady.mx}
  
  \date{}
  \dedicatory{}
  
  
  \begin{abstract}
    It is shown that lattices of a family of split solvable subgroups of $\psl(N+1,\cpx)$ are complex Kleinian using techniques of Lie groups and dynamical systems, also that there exists a minimal limit set for the action of these lattices on the \emph{N} dimensional complex projective space and that there are exactly two maximal discontinuity regions. 
  \end{abstract}

  \maketitle

\section{Introduction and main results}

For $N > 1$ complex Kleinian groups were introduced by Seade and Verjovsky 
in~\cite{seade2001actions},~\cite{Seade2002} as a generalization of
the Kleinian groups of Poincar\'e. A complex Kleinian group is a
discrete subgroup of $\psl(N+1,\cpx)$ which acts properly
discontinuously on some not empty open subset of complex projective space
$\proj^N_\cpx$. To decide whether a subgroup of $\psl(N+1,\cpx)$ is
complex Kleinian is a difficult but important task, 
there have been some advances in complex dimension $N=2$, for example  
in his Ph.D. thesis and in the articles~\cite{Navarrete2006},~\cite{Navarrete2008}, 
Navarrete utilizes the Kulkarni limit set $\Lambda_{kul}$~\cite{kulkarni1978groups}
 in order to find in a canonical way a region of complex projective plane
 where a given discrete subgroup of $\psl(3,\cpx)$ acts properly 
 discontinuously moreover, 
 in $\proj^2_\cpx$ there are 
 many theorems analogous to the classical theory in $\proj^1_\cpx$,
 for example in~\cite{Navarrete2008}
 a classification for the dynamics of the action of cyclic subgroups  
 of $\psl(3,\cpx)$ is found,   
 however even in this dimension there are differences with the
 classical theory of Kleinian groups because there are several  
 non equivalent notions of limit sets
 (cf. \cite{Barrera2018}~\cite{barrera-vargas_cordero_carrillo_2011})
 also every discontinuity region of an infinite complex Kleinian group 
 contains a complex projective line while in the classical theory
 there are groups with only finitely many points in the complement of the
 discontinuity region. 
 In dimension $N > 2$  
  there are a lot of technical difficulties
 to determine whether a given discrete subgroup of $\psl(N+1,\cpx)$ is 
 complex Kleinian, for example 
 the 
 Kulkarni limit set,  
 one of the most important tools in 
 dimension 2,  
 is hard to compute~\cite{cano-limit-set,Cano2017c,cano2017classical}.
 In this article 
 we aim to propose a new approach in order to 
 construct a discontinuity region for a family of discrete 
subgroups of $\psl(N+1,\cpx)$.  
Inspired by the articles~\cite{Mosak1997} and~\cite{barrera2018three}, we adapt techniques of 
dynamical systems and differential geometry to prove the following theorems. 
\begin{theorem}\label{thm:disc-region}
Let $A: \reals \to \gl(N,\reals)$ be a smooth, faithful and closed representation such that each $A(t)$ has no eigenvalue of 
unit length and let $\rho: \semidirectprod \to \gl(N+1,\cpx)$ 
be the representation,
    \begin{equation}\label{eq:g-decomp}
    \rho(b, t) = \begin{pmatrix}
      A(t) & b\\
      0 & 1
    \end{pmatrix}.
  \end{equation}
Assume $G \subset \pgl(N+1,\cpx)$ admits a lift $\tilde G \subset
\gl(N+1,\cpx)$ 
conjugate in $\gl(N+1,\cpx)$ to $\rho(\semidirectprod)$. Let $N_1$
($N_2$) be the cardinality of the 
set of eigenvalues counted with multiplicity of the matrix $A(1)$
inside (outside) the unit disk, then the following holds:
\begin{enumerate}
\item\label{it:thm-1-1}
If $N_i>0$, there is a $G$-invariant open set $\Omega_i \subset
\proj^N_{\cpx}$  where the action
is equivariant to the action on the product 
$G \times X$ given by $(h, g,x) \mapsto (hg, x)$ and 
$X$ is diffeomorphic to the product of an $N_i-1$ sphere and an Euclidean
space. 
  \item\label{it:thm-1-2} If $\Omega \subset \proj^N_{\cpx}$ is
    a $G$-invariant open 
    set where the action is proper, then $\Omega \subset \Omega_i$ for
    some of the sets defined above. 
    
  \item\label{it:thm-1-3} If $G$ admits a lattice $\Gamma$ then $\rho$ is a
    representation into $\gsl(N,\reals)$, there are exactly two open
    sets $\Omega_1$, $\Omega_2$ and the complement $\limitset =
    \proj^N_{\cpx}\setminus (\Omega_1\cup \Omega_2)$ is a closed,
    $G$-invariant set such that
\begin{enumerate}
\item\label{it:thm-1-3-1} The set of points of $\limitset$ with
  infinite isotropy is dense.
\item\label{it:thm-1-3-2} For almost every $z \in \limitset \cap
  \proj^N_{\reals}$ the set of accumulation points of the orbit $\Gamma z$ is
  dense in $\limitset \cap \proj^N_{\reals}$. 
\end{enumerate}
\end{enumerate}
\end{theorem}

As a consequence of the Theorem~\ref{thm:disc-region}, the action on
$\Omega_i$ is proper 
and free provided it is not empty. Moreover, we have the following
corollary.
\begin{corollary}\label{cor:complex-kleinian-lattices}
  If $G \subset \psl(N+1,\cpx)$ is a Lie group satisfying the
  conditions of the Theorem~\ref{thm:disc-region} and it admits a lattice
$\Gamma$, then each open set $\Omega_i$ is a maximal invariant open
set where $\Gamma$ acts properly discontinuously and the quotient space
$\Gamma\setminus \Omega_i$ is a smooth manifold diffeomorphic to
$\Gamma\setminus G \times X$ for a space $X$ 
diffeomorphic to the product of an sphere and an Euclidean space.
\end{corollary}

Hence a lattice of a Lie group as described in the
Theorem~\ref{thm:disc-region} is complex Kleinian. The Theorem and the
Corollary~\ref{cor:complex-kleinian-lattices} are important 
because they give general conditions to construct complex Kleinian
groups independently of the dimension of the projective space and
because they describe the quotient spaces, a task traditionally
difficult to achieve. They also show that the complement $\limitset$ of the
discontinuity regions $\Omega_i$ is a minimal set where lattices of
the group have rich dynamics in the sense described in
part~(\ref{it:thm-1-3}) of the Theorem~\ref{thm:disc-region}. We call
$\limitset$ the limit set of the group since the complement is a union
of open sets where the action of lattices of $G$ is properly
discontinuous.
The second main result is the following theorem, which in
a sense is a partial converse of the
Theorem~\ref{thm:disc-region}. Recall a hyperbolic toral automorphism
is a matrix in $\gsl(N,\integers)$ with no eigenvalue of unit length
and that any such matrix $B$ induces a representation
$\integers \to \gsl(N,\integers)$, such that $k\mapsto B^k$.    

\begin{theorem}\label{thm:limitset}
Let $B \in \gsl(N,\reals)$ be a  hyperbolic toral automorphism and let
$\rho: \integers^N\rtimes_B \integers \to \gsl(N+1,\integers)$ be the
representation
\begin{equation}
\label{eq:rep-integer}
\rho(b,k) = 
\begin{pmatrix}
  B^k & b \\
  0    & 1
\end{pmatrix}.
\end{equation}
If $\Gamma \subset \psl(N+1,\cpx)$ is a discrete subgroup that admits
a lift conjugate in $\psl(N+1,\cpx)$ to $\rho(\integers^N\rtimes_B
\integers)$,
then there are exactly two invariant open sets $\Omega_i \subset
\proj^N_{\cpx}$, $i=1,2$ 
such that 
\begin{enumerate}
  \item\label{thm:1-2-1} The action on each $\Omega_i$ is free, properly
    discontinuous and if $\Omega \subset \proj^N_{\cpx}$ is another
    $\Gamma$-invariant open set where the group acts properly
    discontinuously, then $\Omega \subset \Omega_i$ for
    some $i=1,2$. 
  \item\label{thm:1-2-2} There is a Lie group $G \subset
    \psl(N+1,\cpx)$ such that $\Gamma$ is a lattice of $G$ and for
    each $\Omega_i$ there is a smooth manifold $X_i$ and a fibre
    bundle $\Gamma \setminus \Omega_i \to X_i$ with fibres
    diffeomorphic to $\Gamma \setminus G$ and one of the following
    alternatives hold: 
\begin{enumerate}
\item The bundle is trivial and $X_i$ is contractible to a sphere.
\item $X_i$ is homotopically equivalent to a real projective space. 
\end{enumerate}
\item\label{thm:1-2-3} The limit set
  $\limitset = \proj^N\setminus(\Omega_1\cup \Omega_2)$ satisfies the
  property~(\ref{it:thm-1-3}) of the Theorem~\ref{thm:disc-region}. 
\end{enumerate}
\end{theorem}

The work is organized in the following way. In 
Section~\ref{sec:proof-thm-1} we prove 
Theorem~\ref{thm:disc-region}. In order to do this, 
we study the orbits of the action of $\semidirectprod$ in $\cpx^N$, 
the orbits foliate an open and dense set of $\cpx^N$, to prove 
this we study the action of 
$\reals$ in the imaginary part of $\cpx^N$ by the representation $A(t)$. 
We decompose the imaginary part of $\cpx^N$ in two linear subspaces, 
the stable subspace, where the 
orbits of the action of $\reals$ converge as $t \to \infty$ and 
the unstable subspace  where the orbits diverge to infinity. 
The main result of 
Section~\ref{sec:action-in-complex-plane} is Proposition~\ref{prop:Uplus-difeo}, with this result 
we prove Theorem~\ref{thm:disc-region} in Section~\ref{sec:action-cpN}, our 
tool is Lemma~\ref{lem:stable-diffeo} 
which is proved using techniques of Lyapunov stability. 
The Theorem~\ref{thm:limitset} is proved in
Section~\ref{sec:proof-thm-1.2}, where we determine 
conditions on the representation $A(t)$ for the existence of  
lattices of $\semidirectprod$ (see  
Section~\ref{sec:lattices-split-solvable}). 
For groups $\semidirectprod$ 
admitting a lattice we study the dynamics of the action in the
complement of  
the open sets $\Omega_i$ of Theorem~\ref{thm:disc-region}, 
defined at the end of Section~\ref{sec:action-in-complex-plane}, 
our tool is the fact that hyperbolic toral automorphisms have 
dense orbits in the torus, we show this property extends to lattices 
of $\semidirectprod$ in Lemma~\ref{lem:xg-loxodromic-R} and 
use this result to prove the Theorem~\ref{thm:limitset}.

\section{Proof of theorem~\ref{thm:disc-region}}\label{sec:proof-thm-1}

To prove the Theorem~\ref{thm:disc-region} we will focus 
on the action of $\semidirectprod$ on the projective space induced 
by the representation~\eqref{eq:g-decomp}. Any vector we consider will
be a column vector.  
If $g = (b,t) \in \semidirectprod$ and $[z] \in \proj^N_\cpx$ this action is $g * [z] = [\rho(g)z]$. We distinguish to invariant sets where the 
action is easy to describe. Let $U \subset \proj^N_\cpx$ be the
set of points  with projective coordinates $[z_1:\cdots:z_N:1]$,
$U$ is an invariant open set such that in the chart  $(U,\phi)$, 
$\phi:U\to \cpx^N$,  
  \begin{equation}\label{eq:the-chart-phi}
    \phi([z_1:\ldots:z_N:1]) = (z_1,\ldots, z_N)^T,
  \end{equation}
  and the action of $\semidirectprod$ in
  $U$ is equivariant to the action 
  $(\semidirectprod)
  \times \cpx^N \to \cpx^N$, $((b, t), z) \mapsto 
  A(t) z + b$. On the other hand the complement of $U$
   is an invariant closed set such that 
  for any $g = (b, 0) \in \semidirectprod$ and 
  $[z]  \in \proj^N_\cpx\setminus U$ with projective coordinates 
  $[z_1:\ldots:z_N:0]$
  we have $g*[z] = [z]$, thus points in the complement of $U$ are of 
  infinite isotropy, hence if the action is proper in an open set 
  $\Omega$, this set should be contained in $U$. 
  For the rest of the section we will focus on this action on $U$ 
  identified with the complex space as described above, our aim is to
  determine maximal open sets of $\cpx^N$ for which the action is
  proper.
  \subsection{The action on the complex space}\label{sec:action-in-complex-plane}
  We use the following well known equivalent definition 
  of proper group action~\cite[Prob.~12-19.]{lee_introduction_2011}
\begin{proposition}\label{lem:prop-action-equiv}
  If $G \times \Omega \to \Omega$, $(g,z)\mapsto g*z$ is 
  a differentiable action of a Lie group $G$ into a $G$-invariant subset $\Omega$ of a manifold, 
  the action is proper if and only if for any sequence $(g_n,z_n)$ in $G\times\Omega$ such
  that $z_n$ and $g_n*z_n$ are convergent, there exists a convergent 
  subsequence $g_{n(k)}$. 
\end{proposition}
%
\subsubsection{Globally asymptotically stable linear  systems}\label{subsec:stable-linear-system}
Recall the representation $A: \reals \to \gl(N,\reals)$ has no 
eigenvalue of unit length, let 
$M = \dot A(0)$, then $M$ is a real matrix whose eigenvalues all have 
non zero real part.  
For any $x \in \reals^N$ the  orbits $x(t) = A(t)\,x_0$ are solutions 
to the linear autonomous system,
\begin{equation}\label{eq:linear-autonomous}
    \dot x = M x.
\end{equation}
Assume all the eigenvalues of $M$ have negative real part, 
dynamics of this system is 
well understood from the theory of Lyapunov stability, it is known 
that for this system the origin is exponentially 
stable~\cite[Thm.~8.2]{hespanha2018linear}: there exist positive constants $C$ and 
$\lambda$ such that for every initial condition $x(0) \in \reals^N$ and $t \geq 0$, 
\begin{equation}\label{eq:exponential-stability}
    |x(t)| \leq C\, e^{-\lambda t}\,|x(0)|,
\end{equation}
where $|\cdot|$ is the Euclidean norm, 
moreover, there is a unique symmetric, positive definite matrix 
$P \in \reals^{N\times N}$ such that it is the solution 
to the Lyapunov equation~\cite[Thm.~8.2]{hespanha2018linear},
\begin{equation}\label{eq:lyapunov}
    PM + M^TP = -I.
\end{equation}
\begin{lemma}\label{lem:stable-diffeo}
If all the eigenvalues of the matrix $M\in\reals^{N\times N}$ have negative 
real part, there is an inner product $\langle \cdot, \cdot\rangle$ in 
$\reals^N$ such that if $S$ is the $N-1$ sphere 
$\{x \in \reals^N\,:\,\langle x, x\rangle = 1\}$, the map 
\begin{equation}\label{eq:stable-diffeo}
    \reals \times S \to \reals^N\setminus\{0\}, \qquad (t, x) \mapsto \exp(t M)\,x, 
\end{equation}
is a diffeomorphism.
\end{lemma}

\begin{proof}
Let $\psi: \reals \times S \to \reals^N\setminus\{0\}$ be the 
differentiable map defined in~\eqref{eq:stable-diffeo}, we show 
it is bijective and a local diffeomorphism, thus by the inverse function 
theorem, it is a global diffeomorphism. 
Let $P \in \reals^{N\times N}$ be the solution to the Lyapunov 
equation~\eqref{eq:lyapunov} and define the inner product 
$\langle x, y \rangle = x^T P y$, 
the function $V(x) = \langle x, x \rangle$ is a Lyapunov function for the system, 
this means that 
$V(x) > 0$ for $x \neq 0$ and for any solution $x(t)$ 
to~\eqref{eq:linear-autonomous}, $\frac{d}{dt}\,V(x(t)) < 0$. Let $f(t) = V(x(t))$, 
by~\eqref{eq:exponential-stability}, $\lim_{t\to \infty} f(t) = 0$.  
Since the system~\eqref{eq:linear-autonomous} is linear, 
equation~\eqref{eq:exponential-stability} implies 
\begin{equation}
    \frac{1}{C}\,e^{\lambda t} |x(0)| \leq |x(-t)|, \qquad t \geq 0, 
\end{equation}
hence $\lim_{t\to -\infty} f(t) = \infty$. Since $f(t) = \langle x(t), x(t) \rangle$, we conclude that the curves $x(t)$ intersect exactly once the 
sphere $S$ for any initial condition $x(0)$, this shows $\psi$ is bijective.  
Finally we show $d\psi_{(t,x)}$ is an isomorphism for any $(t,x) \in \reals \times S$, hence it also is a local diffeomorphism. 
  Let $\{v_1,\ldots, v_{N-1}\}$ be a basis for $T_xS$, 
$\langle v_j, x \rangle = 0$, $j = 1,\ldots, N-1$ then  
$\{\partial_t, v_1, \ldots, v_{N-1}\}$ is a basis for  
$T_{(t,x)}(\reals\times S)$, 
since $A(t) = \exp(tM)$, this basis 
is mapped by $d\psi_{(t,x)}$ onto 
$\{A(t)Mx, A(t)v_1,\ldots, A(t)v_{N-1}\}$. On the other hand the  
vector $Mx$ is transversal to $T_xS \subset T_x\reals^N$ because 
\begin{equation}
\langle x, Mx \rangle = \frac{1}{2}\left.\frac{d}{dt}\right\rvert_{t=0}\langle \exp(tM)x, \exp(tM)x \rangle = \frac{1}{2}\left.\frac{d}{dt}\right\rvert_{t=0}f(t) < 0,
\end{equation}
hence $\{Mx, v_1,\ldots, v_{N-1}\}$ is a basis of $\reals^N$, 
since $A(t)$ is an isomorphism we  conclude $d\psi_{(t,x)}$ maps 
a basis of $T_{(t,x)}(\reals \times S)$ onto a basis of 
$T_{\psi(t,x)}\reals^N$ and $\psi$ is a diffeomorphism.
 \end{proof}

\begin{remark}\label{rmk:unstable-diffeo}
If all the eigenvalues of $M$ have positive real part, by Lemma~\ref{lem:stable-diffeo} 
$(t,x) \mapsto \exp(-tM)\,x$ is a diffeomorphism of $\reals \times S$ onto 
$\reals^N\setminus\{0\}$, since the map $(t,x) \mapsto (-t,x)$ is a diffeomorphism  
$\reals \times S \to \reals \times S$, this implies Lemma~\ref{lem:stable-diffeo} also 
holds for $M$ in this case. 
\end{remark}

  \subsubsection{Stable and unstable subspaces of hyperbolic linear systems}\label{subsec:exp-jordan-form}
If $M\in\reals^{N\times N}$ is a matrix such that each eigenvalue of $M$ has non zero real 
part, there is a decomposition 
of the Euclidean space in two invariant subspaces, $\reals^N =
E^s\oplus E^u$,  
the stable and unstable
subspaces respectively, such that if we think of the matrix $M$ as a
linear operator,  
the restriction of $M$ to 
$E^s$ has only eigenvalues with negative real part whereas the restriction 
of $M$ to $E^u$ has only eigenvalues with positive real part. 
If $E^s$ is not trivial, 
by Lemma~\ref{lem:stable-diffeo}
there is a sphere in $E^s$ and 
a well defined diffeomorphism given by the Equation~\eqref{eq:stable-diffeo}
such that the action of $\exp(t M)$ onto $E^s\setminus\{0\}$ is equivariant 
to $\reals
\times (\reals \times S) \to \reals \times S$, $(s, (t,x))
\mapsto (s + t, x)$ and likewise for $E^u$ 
by remark~\ref{rmk:unstable-diffeo}. We will use these 
diffeomorphisms to study the action of $\semidirectprod$. 
%
Our aim is to prove the
existence of two invariant open sets 
$U^s, U^u\subset \cpx^N$ such that
the action of $\semidirectprod$ on each space is proper. We
start by decomposing 
$\cpx^N$ in the direct sum of real subspaces $\reals^N\oplus
iE^s\oplus iE^u$. 
We also define the real 
projections  $\pi_s:\cpx^N\to E^s$ and $\pi_u:\cpx^N\to E^u$. If $E^s$
is not trivial, we define
$U^- = \cpx^N\setminus \operatorname{ker}\pi_s$, if $E^u$ is not
trivial, we define $U^+$ similarly. 
\begin{proposition}\label{prop:Uplus-difeo}
Let $G_A = \semidirectprod$, where 
$A(t)$ is defined as in the Theorem~\ref{thm:disc-region}. 
If $\dim E^s > 0$ there is a diffeomorphism
$U^- \to G_A \times 
\reals^{N - N_s} \times \sphere^{N_s-1}$ where $N_s = \dim E^s$,
$\sphere^{N_s-1} \subset \reals^{N_s}$ is the unit sphere 
and such that the action of $G_A$ is
equivariant to the action 
\begin{equation*}
 G_A \times (G_A \times \reals^{N - N_s} \times \sphere^{N_s-1})
\to G_A \times \reals^{N - N_s} \times \sphere^{N_s-1},   
\end{equation*}
given by $(g, (h, x,y))
\mapsto (gh, x, y)$. 
\end{proposition}
If $\dim E^u > 0$ a similar result holds for
$U^+$. In order to prove Proposition~\ref{prop:Uplus-difeo} we define an 
explicit diffeomorphism and prove that it is in fact a diffeomorphism
in the lemmas~\ref{lem:psi-is-bijective} 
and~\ref{lem:psi-s-diffeo}, the proof
for $U^+$ is identical. Let $\psi^-: G_A \times S \times E^u \to U^-$, 
be the map 
\begin{equation}
\label{eq:psi-plus}
\psi^-(g, x, y) = g*(ix + iy),
\end{equation}
where $S$ is the stable unit sphere as defined
in Lemma~\ref{lem:stable-diffeo}, if $g = (b,t) \in \semidirectprod$,  
then $\psi^-(g,x,y) = b + i(A(t)x + A(t)y)$, where $A(t)x\in
E^s$, $A(t)y\in E^u$. 
\begin{lemma}\label{lem:psi-is-bijective}
$\psi^-$ is bijective.
\end{lemma}

\begin{proof}
Any $z \in U^-$ can be decomposed uniquely as $z
= b + ix + iy$ for some $b \in \reals^N$, $x\in E^s\setminus\{0\}$, 
$y\in E^u$, by Lemma~\ref{lem:stable-diffeo} there is exactly one pair 
$(t,s)\in \reals\times S$ such that $A(t)s = x$. For this pair let 
$y' = A(-t)y$ hence $\psi^-((b,t), s, y') = z$,
thus $\psi^-$ is surjective. If $(g_j,x_j,y_j)$, $j=1,2$ are a pair of points
such that $\psi^-(g_1,x_1,y_1)=\psi^-(g_2,x_2,y_2)$,  where $g_j =
(b_j, t_j)$, then
\begin{equation*}
\label{eq:g2-1-g1-point}
g_2^{-1}g_1*(ix_1 + iy_1) = ix_2 + iy_2,
\end{equation*}
which implies,
\begin{align*}
b_1 - b_2 &= 0, & A(t_1 - t_2) x_1 &= x_2, & A(t_1 - t_2)y_1 &= y_2,
\end{align*}
since $x_j \in S$, $j=1,2$, by Lemma~\ref{lem:stable-diffeo} the second 
condition implies $t_1 - t_2 = 0$ and $x_1 = x_2$ which  
implies $y_1 = y_2$, thus $(g_1,x_1,y_1)=(g_2,x_2,y_2)$
and $\psi^-$ is injective. 
\end{proof}

\begin{lemma}\label{lem:psi-s-diffeo}
$\psi^-$ is a diffeomorphism.
\end{lemma}

\begin{proof}
We show that 
for any $p = (g,x,y) \in G_A\times S \times E^u$ the derivative
$d_p\psi^-$ is an isomorphism 
$T_p(G_A \times S \times E^u) \to T_{\psi^-(p)}U^-$. 
Let $z = \psi^-(p)$, if $g = (b, t)$, then $z = b + i(A(t)x + A(t)y)$, 
note that as real vector spaces,
\begin{align*}
 T_p(G_A \times S \times E^u) = T_gG_A \times T_x S \times T_yE^u, 
 \\
 T_{\psi^-(p)}U^- = T_b\reals^N \times T_{A(t)x} E^s \times T_{A(t)y} E^u,
\end{align*}
then $d\psi^-$
maps $\partial_{b_1},\ldots, \partial_{b_N}$ onto a basis for
$T_b\reals^N$ and since for any $t$ 
the restriction $A(t)|_{E^u}$ is an isomorphism, $d_p\psi^-$ also
maps $T_yE^u$ isomorphically onto $T_{A(t) y}E^u$. 
By Lemma~\ref{lem:stable-diffeo} if $v_1,\ldots, v_{N_s-1}$ is a basis for
$T_xS$, then $d_p\psi^-$ maps the linearly independent vectors
$\partial_t, v_1,\ldots, 
v_{N_s-1}$ to a basis of $T_{A(t)x}E^s$, thus
$d_p\psi^-$ is an isomorphism. By the inverse function theorem,
$\psi^-$ is a diffeomorphism since it is bijective by 
Lemma~\ref{lem:psi-is-bijective}.
\end{proof}

\begin{proof}[Proof of Proposition~\ref{prop:Uplus-difeo}]
Lemmas~\ref{lem:psi-is-bijective} 
and~\ref{lem:psi-s-diffeo}
 prove the first part of the proposition once we 
choose a basis for $E^s$ in order to induce a diffeomorphism 
$S \to \sphere^{N_s - 1}$. 
The final claim of the proposition follows from
the Equation~\eqref{eq:psi-plus}. 
\end{proof}

By Proposition~\ref{prop:Uplus-difeo}, if
$U^-\neq\emptyset$ the
action is proper and free in this set. Another consequence is that
$U^-$ is connected if $N_s > 1$ 
and it has two connected components if and only if $N_s = 1$. Thus for
any lattice $\Gamma \subset \semidirectprod$, the quotient space 
$\Gamma\setminus U^-$ has at most two connected components and it
is diffeomorphic to 
$(\Gamma\setminus \semidirectprod) \times \reals^{N - N_s}
\times S^{N_s-1}$. Similar 
assertions hold for $U^+$. 
We finally show that the sets  $U^{\pm}$ are the only invariant
maximal open sets for the action of $\semidirectprod$.
\begin{proposition}\label{prop:unique-maximal}
If $U \subset \cpx^N$ is an invariant open set where the action is
proper, then $U \subset U^+$ or
$U \subset U^-$.
\end{proposition}

\begin{proof}
If $U^+$ is empty then $U^- = \cpx^N\setminus\{0\}$ and the claim
follows, likewise if $U^-$ is empty. Thus we can assume none of the
open sets $U^{\pm}$ is empty. Since $U^+\cup U^- = \cpx^N\setminus \{0\}$, we can assume in
order to develop a contradiction the existence of $z_1\in U \cap (U^+\setminus U^-)$
and $z_2 \in U\cap(U^-\setminus U^+)$. Since $U$ is invariant under the action of
$\semidirectprod$, we can assume $z_j = iy_j$, $y_j\neq 0$, for
$j=1,2$. Let us define the sequence 
$w_{n} = iA(-n)y_1 + iy_2$,
then $w_{n} \to z_2$ as $n\to \infty$, moreover since $U$ is open,
for $n$ large enough the sequence is contained in $U$. On the other
hand if $g_n = (0,n) \in \semidirectprod$, then
$g_n*w_{n} = iy_1 + i A(n)y_2 \to y_1$ as $n\to \infty$, but the
sequence $g_n$ has no convergent subsequence, a contradiction since
the action in $U$ is proper. 
\end{proof}

%
\subsection{The action in the complex projective
  space}\label{sec:action-cpN} 
%
From Proposition~\ref{prop:Uplus-difeo} we know there is at least one
invariant open subset of $\cpx^N$ such that the action of 
$\semidirectprod$ is proper. In this section we prove this set
is identified with an open maximal set in the complex projective space where
the action is proper and conclude the first assertions of the
Theorem~\ref{thm:disc-region}.

\begin{proof}[Proof of Theorem~\ref{thm:disc-region}, parts~(\ref{it:thm-1-1}) and~(\ref{it:thm-1-2})]
Let $\tilde G$ be a lift of the group $G \subset \pgl(N+1,\cpx)$ 
conjugate by a matrix $\sigma \in \gl(N+1,\cpx)$  to the image of the 
representation $\rho$ given by Equation~\eqref{eq:g-decomp}:
For any $g \in \tilde G$ there is a $(b,t)\in\semidirectprod$
such that $\sigma^{-1}g\sigma = \rho(b,t)$, moreover $\sigma$ acts on
projective space as
$\sigma [z] = [\sigma z]$, $[z] \in\proj^N_{\cpx}$. 
\par{Proof of~(\ref{it:thm-1-1}):} Recall the chart $(U,\phi)$ is defined
by~\eqref{eq:the-chart-phi}, if
$U^-\neq \emptyset$ let us define the open set
$\Omega^- = \sigma\phi^{-1}(U^-)$,  
if $U^+\neq \emptyset$ we define $\Omega^+$
similarly. Each of the sets $\Omega^\pm$ is diffeomorphic to $U^{\pm}$
which on the other hand is diffeomorphic to a product of the form
$G_A\times \reals^{N-N_{\pm}}\times \sphere^{N_{\pm}-1}$ by
Proposition~\ref{prop:Uplus-difeo}, moreover $\Omega^{\pm}$ is
$G$-invariant because $U^\pm$ is invariant under the action of
$\semidirectprod$ and the $G$-action on $\Omega^\pm$ is equivariant to
the action of $\semidirectprod$ on $U^\pm$ given  
in Proposition~\ref{prop:Uplus-difeo}, defining $X$ as the preimage of 
$ \{e\}\times \reals^{N-N_{\pm}}\times \sphere^{N_{\pm}-1}$ proves the
first part of the Theorem~\ref{thm:disc-region}.
\par{Proof of~(\ref{it:thm-1-2}):} If $\Omega \subset \proj^N_{\cpx}$
is a $G$-invariant open set where the group acts properly, then
$\Omega \cap (\proj^N_\cpx\setminus \sigma U) = \emptyset$ 
because each point in $\proj^N_\cpx \setminus \sigma U$ has infinite 
isotropy, hence $\phi(\sigma^{-1} \Omega) \subset \cpx^N$ is an open
invariant set where the action of
$\semidirectprod$ is proper, by  
Proposition~\ref{prop:unique-maximal} 
$\phi(\sigma^{-1} \Omega) \subset U^\pm$,
hence $\Omega \subset \Omega^\pm$. 
\end{proof}

  \section{Proof of theorem~\ref{thm:limitset}}\label{sec:proof-thm-1.2}
\subsection{The lattices of a family of split solvable Lie groups}\label{sec:lattices-split-solvable}

Mosak and Moskowitz described the lattices 
of the semi-direct product $\semidirectprod$ in 
the case where the
matrix $A$ is diagonal and 
$A(t)\in \gsl(N,\reals)$ for all $t$~\cite{Mosak1997}.
If  $A:\reals \to \gl(N, \reals)$ is a faithful and 
closed representation there is  a matrix $M \in \gl(N,\reals)$
such that $A(t) = \exp(t\,M)$, by the result
of~\cite[p.~3]{raghunathan1972} the  
nilradical of $\semidirectprod$ is the identity
component of  the group $\{g\in \semidirectprod \,:\,
Ad_g\,\text{is unipotent}\}$. Since $A(t)$ is closed we can 
compute explicitly the adjoint representation of a matrix  
$g=\rho(b,t)$ in
$\gl(N+1,\reals)$ as given by~\eqref{eq:g-decomp}: For any 
 $(N+1)\times (N+1)$ real matrix $X$  
in the Lie algebra of the image of $\rho$, $Ad_g(X)=
gXg^{-1}$. Moreover, for any such  $X$ there is exactly one
$s\in\reals$ and $p\in\reals^N$ such that $X$ is a block matrix,
\begin{equation}
\label{eq:X-lie-algebra}
X = 
\begin{pmatrix}
  s\,M & p\\
  0     & 0
\end{pmatrix},
\end{equation}
hence, an element $g\in \semidirectprod$ is unipotent provided
the only $\lambda \in \cpx$ such that $Ad_gX = \lambda\,X$ has a 
non trivial solution is $\lambda = 1$. 
By a direct computation, this condition is equivalent
to require that for some $s \in \cpx$ and $p \in \cpx^N$,
\begin{align}
\label{eq:adg-eigenvalue-problem}
sM = \lambda sM, && -sMb + A(t)p = \lambda p.
\end{align}
Note that $M \neq 0$ because $A(t)$ cannot be a constant 
function.  
If $\lambda \neq 1$, by~\eqref{eq:adg-eigenvalue-problem} $s=0$ and since 
$X$ is not trivial, $p \neq 0$, hence  
$\lambda$ is 
an eigenvalue of $A(t)$ with eigenvector $p$. 
Thus the only possible eigenvalues of $Ad_g$ are $1$ 
and the eigenvalues of $A(t)$. 
\begin{proposition}\label{prop:semidirect-nilradical}
If $M$ has a nonzero eigenvalue, the nilradical of $\semidirectprod$ is
$\reals^N$, otherwise $\semidirectprod$ is nilpotent. 
\end{proposition}
\begin{proof}
Firstly we show $\reals^N$ is contained in the
nilradical. By~\eqref{eq:adg-eigenvalue-problem}, for $t=0$,
$b \in \reals^N$ we can take $X=X(0,p)$ in  
equation~\eqref{eq:X-lie-algebra} where $p \in \reals^N$ is any 
vector such that $p\neq 0$. 
This matrix will be an
eigenvector of $Ad_g$ with eigenvalue $\lambda = 1$, hence the
claim. If $M$ has a nonzero eigenvalue $d$, 
then $\lambda = \exp(td)$ 
is an eigenvalue of $A(t)$ such that $\lambda \neq 1$ for $t \neq 0$. 
If $p \in
\cpx^N$ is a non zero eigenvector of $A(t)$ for this eigenvalue,
equation~\eqref{eq:adg-eigenvalue-problem} shows $X(p,0)$ is an
eigenvector of $Ad_g$ with the same eigenvalue, thus $Ad_g$ is not
unipotent for $t\neq 0$, hence the nilradical is also contained in
$\reals^N$. Finally if the only eigenvalue of $M$ is $0$, 
$A(t)$ is unipotent for all $t$, thus $Ad_g$ is also
unipotent for all $g \in \semidirectprod$ and the nilradical is
the whole group.
\end{proof}

As mentioned earlier, Mosak and Moskowitz studied  lattices of
$\semidirectprod$ for a diagonal representation $A(t)$ in
$\gsl(N,\reals)$. In preparation for 
the next proposition, we reproduce the relevant results found  
in~\cite{Mosak1997} that work for a general representation. 
Assume $M$ is invertible, hence the nilradical of
$\semidirectprod$ is $\reals^N$ by 
Proposition~\ref{prop:semidirect-nilradical}. If $\Gamma \subset
\semidirectprod$ is a lattice, then $L = \Gamma \cap \reals^N$ is a lattice in $\reals^N$~\cite{raghunathan1972}, thus there is a matrix
$\sigma \in \gl(N,\reals)$ such that $L = \sigma^{-1}\integers^N$. On
the other hand, $\Gamma \cap \reals$ is a lattice of
$\reals$ hence there is a $h > 0$ such that
$\Gamma \cap \reals = h\integers$. Let
$\mathcal{M}= L\rtimes_A h\integers$, $\mathcal{M}$ is a lattice in
$\Gamma$. Let us denote by $g_1 \cdot g_2$ the group operation in $\semidirectprod$, 
since $(0,h)\in \Gamma$, for any
$(p,0) \in L$, we have $(A(h)p,0) = (0,h)\cdot (p,0) \cdot (0,-h) \in \Gamma$, on the other hand $(A(h)p,0) \in \reals^N$, hence $A(h)L \subset L$, by
a similar argument $A(-h)L\subset L$, implying $L \subset A(h)L$,
whence $A(h)L = L$ and $B(h) = \sigma A(h) \sigma^{-1}$ is an homomorphism
of $\integers^N$, but the previous argument also shows $B(h)^{-1} =B(-h)$ is 
another homomorphism, hence $B(h) \in \gsl(N,\integers)$. With these 
previous steps, 
we arrive at the following proposition,
\begin{proposition}\label{prop:lattice}
If $A:\reals \to \gl(N,\reals)$ is a faithful, closed representation
with no eigenvalue of unit length, 
$\semidirectprod$ has a lattice $\Gamma$ if and only if 
there are a matrix $\sigma \in \gl(N, \reals)$ and 
  $h > 0$ such that,
  \begin{equation}\label{eq:lattice-condition}
      \sigma A(h)\sigma^{-1} \in \gsl(N,\integers), 
  \end{equation}
  moreover $\sigma^{-1}\integers^N \rtimes_A h\integers \subset \Gamma$. 
\end{proposition}
\begin{proof}
If $\semidirectprod$ has  a lattice, by the previous
computations there are $\sigma \in \gl(N,\reals)$ and $h > 0$ such that $A(h) =
\sigma^{-1}B\sigma$ for some matrix $B\in\gsl(N,\integers)$ such that 
$\sigma^{-1}\integers^N \rtimes_A h\integers \subset \Gamma$, 
hence~\eqref{eq:lattice-condition} holds. 
If on the other hand~\eqref{eq:lattice-condition} holds, 
$\sigma^{-1}\integers^N\rtimes_A h\integers$ is a 
discrete and co-compact subgroup of 
$\semidirectprod$, hence it is a lattice of $\semidirectprod$.  
\end{proof}

\subsection{Hyperbolic affine actions in projective space}\label{sec:maximal-limit-set}

As a consequence of the Proposition~\ref{prop:lattice}, if
$\semidirectprod$ has a lattice then $A(t)$ is a representation in
$\gsl(N,\reals)$. In this section $A(t)\in\gsl(N,\reals)$ for all $t$,
by Proposition~\ref{prop:unique-maximal} we know the sets $U^{\pm}$
are maximal open sets where the action of $\semidirectprod$ is proper. 
Recall $\phi$ is the chart defined by 
the Equation~\eqref{eq:the-chart-phi}. If $\Gamma \subset
\semidirectprod$ is a lattice, then the sets $\Omega^{\pm} =
\phi^{-1}(U^{\pm})$ are discontinuity regions for the action of
$\Gamma$. Let $\limitset^{\pm} = \proj^N_{\cpx}\setminus
\Omega^{\pm}$, each of these sets is a limit set of $\semidirectprod$
for the
action of $\Gamma$ in the sense that it is the complement of a maximal
discontinuity region, we aim to study the dynamics of the action of 
$\Gamma$ in $\limitset = \limitset^+ \cap \limitset^-$. 
If $\exp(t M) \in \gsl(N,\reals)$ for all $t$, $M$ has no purely imaginary eigenvalue  and $\Gamma$ 
is a lattice of $\semidirectprod$, by Proposition~\ref{prop:lattice} there 
is a matrix 
 $\sigma \in \gl(N,\reals)$ and a constant $h>0$ such that 
$\mathcal{L} =\sigma^{-1}\integers^N \rtimes_A h\integers \subset \Gamma$. Let 
$p \in \integers^N$, $n \in \integers$,  if we change the coordinates 
of $\cpx^N$ by means of the matrix $\sigma^{-1}$, an element 
$g=(p,n) \in \mathcal{L}$ 
acts on $z \in \cpx^N$ as $g*z = B^nz + p$, where 
$B = \sigma^{-1}\exp(hM)\sigma \in\gsl(N,\integers)$. 
The linear action of the cyclic group generated by $B$ onto $\reals^N$
descends to an action of the $N$-torus $\torus^N = \reals^N/\integers^N$
as $(B^n, [x]) \mapsto [B^n x]$, $[x] \in \torus^N$.
The set of points $[x] \in \torus^N$ with dense orbits is also
dense by~\cite[Thm.~1.11]{walters2000introduction}. 
As mentioned earlier, points in $\proj^N_{\cpx}\setminus U$ have infinite
isotropy. On the other hand, we can identify $U$ with $\cpx^N$ by means of
the chart $(U,\phi)$. In this chart a point $z \in \cpx^N$ has infinite
isotropy if and only if for some $b \in \integers^N$,
$n\in\integers\setminus \{0\}$,
\begin{align}\label{eq:x-b-k}
z = (I - B^n)^{-1}\,b.
\end{align}
Note that for any pair $(b,n) \in \integers^N\rtimes_B\integers$ the
solution to~\eqref{eq:x-b-k} is real, 
we state the following two lemmas,
\begin{lemma}\label{lem:GA-Xs-dense}
  For almost every $x^* \in \reals^N$, the orbit of the 
action $(\integers^N \rtimes_B \integers) \times \reals^N \to \reals^N$, 
$((b,n),x^*) \mapsto B^nx^* + b$, is dense.
\end{lemma}

\begin{proof}
For almost every $x^{*} \in
\reals^N$ the orbit $[B^n   
x^{*}]$ is dense in $\torus^N$. Let $x \in \reals^N$ 
be any other point which we assume is in the interior
of the fundamental domain of the lattice generated by the canonical basis 
$\left\{e_1,\ldots,e_N\right\} \subset \reals^N$.    
For any $n >0$ there is a $b_n \in \integers^N$ such that 
$x$ and $B^nx^* + b_n$ are in the same 
tile, moreover since the action of the cyclic group $\langle B \rangle$ 
is dense in the torus, there is
an integer sequence $n_j$ such that $[B^{n_j}x^{*}]\to
[x]$ which implies that 
$B^{n_j}x^{*}+b_{n_j} \to x$ in $\reals^N$ as $j \to \infty$. 
  Since the set of interior points in the tile is dense, this proves
  the lemma.   
\end{proof}
\begin{lemma}\label{lem:I-Ak-inv-bounded}
Let $\abs{\cdot}$ be a norm in $\cpx^N$ and let $\|\cdot\|$ be the associated 
operator norm, for any matrix $A \in \gl(N,\reals)$ such that no eigenvalue is
of unit length the sequence $||(I - A^n)^{-1}||$ is bounded.
\end{lemma}

\begin{proof}
Firstly, let us assume the spectrum of $A$ is contained in the open unit
disk, then $A^n\to 0$ as $n \to \infty$. By
continuity of the matrix inverse function, $\|(I-A^n)^{-1}\| \to \|I\|
= 1$.  On the other hand, if no eigenvalue of $A$ is contained in the
 disk, then the matrix $A^{-1}$ satisfies our first
assumption, hence $\|(I - A^n)^{-1}\| = \|(I - A^{-n})^{-1}A^{-n}\|
\to 0$ as $n \to \infty$. Finally, for a general matrix we can
decompose $\cpx^N$ in two $A$-invariant subspaces, $\cpx^N = X\oplus Y$,
such that as a linear operator, the restriction $A|_X$ satisfies the
first condition whereas $A|_Y$ satisfies the second. In this case we
can introduce a second norm $|\cdot|_{\oplus}$ in $\cpx^N$ such that if $z =
x + y$, $x \in X$, $y \in Y$, then $|z|_\oplus = |x|_\oplus + |y|_\oplus$. Let
$\|\cdot\|_\oplus$ be the corresponding operator norm, the following bounds
hold,  
\begin{align*}
  \|(I - A^n)^{-1}\|_\oplus &= \sup_{|z|_\oplus=1}
                         \left\lvert(I-A^n)^{-1}z\right\rvert_\oplus\\
                       &\leq \sup_{|x|_\oplus\leq 1}
                         \left\lvert(I-A^n)|_X^{-1}x\right\rvert_\oplus 
                         + \sup_{|y|_\oplus\leq 1}
                         \left\lvert(I-A^n)|_Y^{-1}y\right\rvert_\oplus\\ 
  &\leq \|(I - A^n)|_X^{-1}\|_\oplus + \|(I - A^n)|_Y^{-1}\|_\oplus,
\end{align*}
the right hand side of the last inequality is convergent as $n\to
\infty$, hence $\|(I - A^n)^{-1}\|_\oplus$ is a bounded sequence. Since any
two norms in $\gl(N,\reals)$ are equivalent, this proves the Lemma. 
\end{proof}

\begin{lemma}\label{lem:xg-loxodromic-R}
The set of solutions $x(b,n)$ to equation~\eqref{eq:x-b-k}
is dense in $\reals^N$.
\end{lemma}
\begin{proof}
  Let $x^{*} \in \reals^N$, the 
  Lemma~\ref{lem:GA-Xs-dense} defines a sequence
$y{(b,n)} = B^nx^{*} + b$ 
which we know is
dense in $\reals^N$ for almost every $x^*$. A short
computation shows, 
\begin{align}
\label{eq:25}
x(b,n) - x^{*} = (I -
  B^n)^{-1}(y(b,n) - x^{*}), \qquad n \neq 0. 
\end{align}
Denoting by $|\cdot|$ any norm in $\cpx^N$ and by $\|\cdot\|$ the
corresponding operator norm in $\gl(N, \cpx)$, 
equation~\eqref{eq:25} implies, 
\begin{align}
|x(b,n) - x^{*}| &\leq \|(I -
  B^n)^{-1}\|\,|y(b,n) - x^{*}|,\nonumber\\
 &\leq C\,|y(b,n) - x^{*}|, \label{eq:31}
\end{align}
where the positive constant $C$ is given by
Lemma~\ref{lem:I-Ak-inv-bounded}. 
Since $\{y(b,n) \,:\, b \in \integers^N, n \in \integers^+\}$ is dense
in $\reals^N$, 
there is a convergent subsequence $y{(b_j,n_j)} \to
x^{*}$ as $j\to\infty$, hence we can find a subsequence $x(b_j,n_j)
\to x^{*}$ as $j \to \infty$. This proves the lemma because the set of
points $x^{*}$ for which the convergence holds is dense in
$\reals^N$.   
\end{proof}

Proposition~\ref{prop:discrete-u-max} is similar to 
Proposition~\ref{prop:unique-maximal}, we make a few modifications in the proof.

\begin{proposition}\label{prop:discrete-u-max}
If $U \subset\cpx^N$ is an open $\integers^N\rtimes_B\integers$
invariant set  where the action is proper, then $U \subset U^+$ or $U
\subset  U^-$.
\end{proposition}

\begin{proof}
  Let $E^s$ and $E^u$ denote the stable and unstable subspaces of $B$
  and let us assume in order to develop a contradiction the existence
  of two points,
  \begin{align*}
    z_1 &\in U \cap (U^-\setminus U^+) \subset (\reals^N\oplus
          iE^s)\setminus \reals^N\oplus 0,\\
    z_2 &\in U \cap (U^+\setminus U^-) \subset (\reals^N\oplus
          iE^u)\setminus \reals^N\oplus 0,
  \end{align*}
 such that $z_j = x_j + iy_j$, where $y_1\in E^s\setminus\{0\}$
and $y_2\in E^u\setminus\{0\}$. Since $U\cap (\reals^N\oplus 0)$ is relatively
open, by the Lemma~\ref{lem:GA-Xs-dense} we may assume there is a
sequence $g_j = (b_j, n_j) \in \integers^N\rtimes_B\integers$ such
that $g_j*x_1 \to x_2$ as $j\to \infty$, moreover, the proof of the Lemma shows that we
may assume $n_j \to \infty$ as $j \to \infty$. Let
$w_j = g_j*x_1 + iB^{n_j}y_1 + iy_2$, then for $j$ sufficiently large,
$w_j \in U$ because $w_j \to z_2$. On the other hand, 
$g_j^{-1}*w_j= x_1 + iy_1 + i B^{-n_j}y_2 \to z_1$
as $j \to \infty$. This is a contradiction because the sequence
$g_j^{-1}$ has no convergent subsequence in
$\integers^N\rtimes_B\integers$.
\end{proof}

\begin{proof}
  [Proof of Theorem~\ref{thm:disc-region}-(\ref{it:thm-1-3})]
  If $\Gamma \subset \semidirectprod$ is a lattice, by 
Proposition~\ref{prop:lattice} there are matrices
$B \in\gsl(N,\integers)$,  
$\sigma \in \gl(N,\reals)$ and $h > 0$ such that 
$\mathcal{L} = \sigma^{-1}\integers \rtimes_B h\integers$ is another
lattice  contained in $\Gamma$. By the Lemma~\ref{lem:xg-loxodromic-R}, for
almost every $x \in \limitset \cap
\proj^N_\reals$ the $\mathcal{L}$-orbits are dense, hence the same
claim holds for $\Gamma$.
The second claim follows from the fact that $\mathcal{L}$ has  
finite index in $\Gamma$ and the closure of the set of points with
infinite isotropy with respect to $\mathcal{L}$ is $\limitset$.
\end{proof}

\begin{proof}[Proof of Theorem~\ref{thm:limitset}]
  For any hyperbolic toral automorphism $B \in \gsl(N,\integers)$
  by~\cite[Thm.~1]{culver_existence_1966} there is a real $N\times N$
  matrix $M$ such that either $B = \exp(M)$ or $B^2 = \exp(M)$. In
  any case we define $A(t) = \exp(tM)$, $t\in\reals$ and the group
  \begin{equation}
    \label{eq:disconnected-lift}
    \tilde G = \left\{
      \begin{pmatrix}
        C & p\\
        0 & 1
      \end{pmatrix}\in\gsl(N+1,\reals)\,:\,
      C = A(t)\text{ or } C = BA(t)\,\text{for some } t \in \reals
    \right\},
  \end{equation}
  then for the representation $\rho$ defined
  in~(\ref{eq:rep-integer}), 
  $\rho(\integers^N\rtimes_B\integers)$ is a lattice of $\tilde G$.
  \par{\emph{(\ref{thm:1-2-1}).}} Let $G$ be the projection of
  $\tilde G$ to $\psl(N+1,\cpx)$, and let $G_0 \subset G$ be the connected
  component of the identity which admits a lift to $\semidirectprod$. 
  The matrices $B$ and $B^2$ have
  eigenvalues at both sides of the unit disk, hence by the
  Theorem~\ref{thm:disc-region}-(\ref{it:thm-1-1})-(\ref{it:thm-1-2})
  there are two open sets $\Omega_1$,
  $\Omega_2$ where $G$ acts properly and freely. From the proof
  of the Theorem~\ref{thm:disc-region}, we know $\Omega_i\subset U$
  for $(U,\phi)$ the canonical chart defined
  in~\eqref{eq:the-chart-phi}. A vector $z \in \cpx^N$
  belongs to $\phi(\Omega_i)$ if and only if $z = x + iy$, and $y \in
  (E^s\setminus \{0\}) \cup (E^u\setminus \{0\})$, where $E^s$ and
  $E^u$ are the stable and unstable subspaces of $A(1)$. The
  stable and unstable subspaces of $B$ and $B^2$ are the same, hence each
  $\Omega_i$ is $\Gamma$-invariant in either case.
$\Gamma$ acts 
properly discontinuously on each $\Omega_i$ because the action of
$G_0$ is proper hence  
$\integers^N\rtimes_{B^2} \integers$ acts properly discontinuously and 
is a subgroup of finite index  of $\integers^N\rtimes_B\integers$. The
final claim follows from  Proposition~\ref{prop:discrete-u-max} applied in 
the chart $(U,\phi)$.
\par{\emph{(\ref{thm:1-2-2})}.} Note that each $\Omega_i$ is also
$G$-invariant and since the index $[G:G_0]$ is finite, $G$ acts
properly on $\Omega_i$, moreover, the action is free, hence the
canonical projection $\Omega_i\to G\setminus \Omega_i$ is a principal
bundle with fibres diffeomorphic to $G$. If $G$ is connected, by the
Theorem~\ref{thm:disc-region}-(\ref{it:thm-1-1}) this bundle is
trivial and the base space is contractible to a sphere. In the second case,
note that $G_0$ is normal in $G$ and 
$G_0\setminus G \cong\integers_2$. If $g \in G$ has non trivial class
in $G_0\setminus G$, then it determines an involution $f_g: G_0\setminus
\Omega \to G_0\setminus \Omega$, $f_g(G_0x) = G_0(g*x)$, hence 
an action of $\integers^2$ into $G_0\setminus \Omega$. The map
$G_0\setminus \Omega \to G\setminus \Omega$, $G_0x\mapsto Gx$ projects
to a diffeomorphism $\integers^2\setminus (G_0\setminus\Omega_i)\cong
G\setminus \Omega_i$. We aim to show by means of this diffeomorphism
that $G\setminus \Omega_i$ is homotopically equivalent to a real
projective space. Assume without loss of generality  
$G_0\setminus \Omega_i$ diffeomorphic to
$\semidirectprod \setminus U^-$ and that $g$ lifts to 
$\tilde g = \left(\begin{smallmatrix}
    B & 0\\
    0 & 1
\end{smallmatrix}\right)$. 
Let $\psi: S^{N_--1}\times E^u \to
G_0\setminus \Omega_i$, $\psi(x,y) = G_0*(ix + iy)$ be the
diffeomorphism induced
by the Proposition~\ref{prop:Uplus-difeo}. Recall for any trajectory
$A(t) x$ in the stable subspace, there is exactly one intersection
with the stable sphere $S$, hence for any $x \in E^s$ there is exactly
one parameter $t(x) \in \reals$ such that $A(t(x)) x \in S$. By the
implicit function theorem, the 
function $t(x)$ is smooth, moreover, if $F: S \to \tilde G$ is
the map $F(x)=A(t(\tilde g * x))\tilde g$, then  
$\tilde g * \psi(x,y) = G_0* (i F(x)x + i F(x)y)$. Let $f_1: S \to S$ be
the map $f_1(x) = F(x)*x$ and $f_2: S \to \gl(E^u)$ be the map $f_2(x)\,y
= F(x)*y = A(t(B x))By$, then up to diffeomorphism, $\tilde g$ acts on
$S \times E^u$ as $\tilde g * (x, y) = (f_1(x), f_2(x)\,y)$, at the
same time, $\tilde g$ also acts on $S$ as  $\tilde g * x = f_1(x)$ and
$S\times E^u$ is homotopically equivalent to the sphere $S \times
\{0\}$ by the homotopy $h_s(x,y) = (x, sy)$, $s \in [0, 1]$, since
$f_2$ depends only on $x \in S$ and is linear, $h_s$ induces another
homotopy $\hat h_s: \integers^2\setminus (S \times E^u) \to
\integers^2\setminus (S\times\{0\})$ as $\hat h_s([(x,y)])=
[h_s(x,y)]$. Thus the quotient space $\integers^2\setminus
(G_0\setminus \Omega)$ is homotopically equivalent to the quotient
$\integers^2\setminus S$ of an sphere by an involution. By a known result~\cite{olum1953mappings} this implies that it is
homotopically equivalent to $\proj^{N_s-1}_{\reals}$. 
  \par{\emph{(\ref{thm:1-2-3}).}} If $G$ is connected, the result is
  immediate from the Theorem~\ref{thm:disc-region}, otherwise, the
  subgroup $\Gamma'$ satisfies the conclusion of the Theorem~\ref{thm:disc-region}.  Thus, also with respect to $\Gamma$  points of
  infinite isotropy in $\limitset$ are
  dense and for
  almost any $z \in \limitset \cap \proj^N_{\reals}$ the orbits are
  dense. 
\end{proof}

\section*{Acknowledgments}
The research of W. Barrera, R. Garc\'ia and J. P. Navarrete have been supported by the CONACYT, Proyecto Ciencia de Frontera 2019--21100 via the Faculty of Mathematics, UADY.

\bibliographystyle{amsplain}
\bibliography{main}

\end{document}